\documentclass[10pt]{amsart}
\usepackage{amssymb}
\usepackage{epsfig}
\usepackage{url}
\usepackage{setspace}
\theoremstyle{plain}

\newtheorem{thm}{Theorem}[section]
\newtheorem{cor}[thm]{Corollary}

\newtheorem{rem}[thm]{Remark}
\newtheorem{ques}[thm]{Question}

\newtheorem{exam}[thm]{Example}

\def\bbb{\mathbb}

\renewcommand{\phi}{\varphi}

\newcommand{\N}{\bbb{N}}
\newcommand{\Z}{\bbb{Z}}
\newcommand{\Q}{\bbb{Q}}

\newcommand{\C}{\bbb{C}}

\begin{document}

\title[A note on Sierpi\'{n}ski problem]{A note on Sierpi\'{n}ski problem related to triangular numbers}
\author{Maciej Ulas}


\begin{abstract}
In this note we show that the system of equations
\begin{equation*}
t_{x}+t_{y}=t_{p},\quad t_{y}+t_{z}=t_{q},\quad t_{x}+t_{z}=t_{r},
\end{equation*}
where $t_{x}=x(x+1)/2$ is a triangular number, has infinitely many
solutions in integers. Moreover we show that this system has
rational three-parametric solution. Using this result we show that
the system
\begin{equation*}
t_{x}+t_{y}=t_{p},\quad t_{y}+t_{z}=t_{q},\quad
t_{x}+t_{z}=t_{r},\quad t_{x}+t_{y}+t_{z}=t_{s}
\end{equation*}
has infinitely many rational two-parametric solutions.
\end{abstract}

\maketitle

\section{Introduction}\label{Section1}

By a triangular number we call the number of the form
\begin{equation*}
t_{n}=1+2+\ldots +n-1 + n=\frac{n(n+1)}{2},
\end{equation*}
where $n$ is a natural number. Triangular number can be interpreted
as a number of circles  necessary to build equilateral triangle with
a side of length $n$. There are a lot of papers related to the
various types of diophantine equations containing triangular numbers
and various generalizations of them
\cite{Sier1,Sier2,Sier3,Sier4,Sier5}. One of my favourite is a
little book \cite{Sier4} written by W. Sierpi\'{n}ski.

On the page 33. of his book W. Sierpi\'{n}ski stated an interesting
question related to the triangular numbers. This question is
following:

\begin{ques}\label{ques1}
Is it possible to find three different triangular numbers with such
a property that sum of any two is a triangular number? In other
words: is it possible to find solutions of the system of equations
\begin{equation}\label{eq1}
t_{x}+t_{y}=t_{p},\quad t_{y}+t_{z}=t_{q},\quad t_{z}+t_{x}=t_{r},
\end{equation}
in positive integers $x,y,z,p,q,r$ satisfying the condition $x<y<z$?
\end{ques}

In the next section we give all integer solutions of the system
(\ref{eq1}) satisfying the condition $x<y<z<1000$ and next we
construct two one-parameter polynomial solutions of our system
(Theorem \ref{thm1}).

In section \ref{Section3} we change perspective a bit and ask about
rational parametric solutions of our problem. Using very simple
reasoning we are able to construct rational parametric solution with
three variable parameters (Theorem \ref{thm2}).

Finally, in the last section we consider the system of equations
\begin{equation*}
t_{x}+t_{y}=t_{p},\; t_{y}+t_{z}=t_{q},\;
t_{z}+t_{x}=t_{r},\;t_{x}+t_{y}+t_{z}=t_{s}.
\end{equation*}
We give some integer solutions of this system. Next, we use
parametrization obtained in the Theorem \ref{thm2} to obtain
infinitely many rational solutions depending on two parameters. In
order to prove this we show that certain elliptic curve defined over
the field $\Q(u,v)$ has positive rank.

\section{Integer solutions of the system $t_{x}+t_{y}=t_{p},\; t_{y}+t_{z}=t_{q},\; t_{z}+t_{x}=t_{r}$}\label{Section2}
In order to find integer solutions of the system (\ref{eq1}) we use
the computer. We are looking for solutions satisfying the condition
$x<y<z<1000$. We find 44 solutions in this range. Results of our
search are presented in Table 1.

\begin{equation*}
\begin{array}{ccc}
  \begin{tabular}{|c|c|c|c|c|c|}
     \hline
     $x$ & $y$ & $z$ & $p$ & $q$ & $r$ \\
     \hline
     \hline
        9& 13& 44& 16& 46& 45\\
        14& 51& 104& 53& 116& 105\\
        20& 50& 209& 54& 215& 210\\
        23& 30& 90& 38& 95& 93\\
        27& 124& 377& 127& 397& 378\\
        35& 65& 86& 74& 108& 93\\
        35& 123& 629& 128& 641& 630\\
        41& 119& 285& 126& 309& 288\\
        44& 245& 989& 249& 1019& 990\\
        51& 69& 104& 86& 125& 116\\
        54& 143& 244& 153& 283& 250\\
        62& 99& 322& 117& 337& 328\\
        65& 135& 209& 150& 249& 219\\
        66& 195& 365& 206& 414& 371\\
        74& 459& 923& 465& 1031& 926\\
        76& 90& 144& 118& 170& 163\\
        77& 125& 132& 147& 182& 153\\
        77& 125& 207& 147& 242& 221\\
        83& 284& 494& 296& 570& 501\\
        105& 170& 363& 200& 401& 378\\
        105& 363& 390& 378& 533& 404\\
        105& 551& 924& 561& 1076& 930\\
     \hline
   \end{tabular}
   & &
 \begin{tabular}{|c|c|c|c|c|c|}
     \hline
     $x$ & $y$ & $z$ & $p$ & $q$ & $r$ \\
     \hline
     \hline
   114& 429& 650& 444& 779& 660\\
         131& 174& 714& 218& 735& 726\\
         131& 245& 714& 278& 755& 726\\
         135& 154& 531& 205& 553& 548\\
         161& 260& 924& 306& 960& 938\\
         170& 469& 755& 499& 889& 774\\
         189& 305& 406& 359& 508& 448\\
         216& 390& 854& 446& 939& 881\\
         230& 741& 870& 776& 1143& 900\\
         237& 527& 650& 578& 837& 692\\
         245& 714& 989& 755& 1220& 1019\\
         252& 272& 702& 371& 753& 746\\
         278& 370& 594& 463& 700& 656\\
         286& 405& 494& 496& 639& 571\\
         293& 390& 854& 488& 939& 903\\
         299& 441& 560& 533& 713& 635\\
         350& 629& 781& 720& 1003& 856\\
         476& 634& 665& 793& 919& 818\\
         581& 774& 935& 968& 1214& 1101\\
         588& 645& 689& 873& 944& 906\\
         609& 779& 923& 989& 1208& 1106\\
         714& 798& 989& 1071& 1271& 1220\\
\hline
\end{tabular}
\end{array}
\end{equation*}
\begin{center}
{\sc Table 1}
\end{center}

\bigskip

The above solutions show that the answer on Sierpi\'{n}ski's
question is easy. However, due to abundance of solutions it is
natural to ask if we can find infinitely many solutions of the
system (\ref{eq1}).

\begin{thm}\label{thm1}
The system of equations {\rm(\ref{eq1})} has infinitely many
solutions in integers.
\end{thm}
\begin{proof}
In order to find infinitely many solutions of our problem we
examined the Table 1. We find that for any $u\in\N$ the values of
polynomials given by
\begin{equation*}
\begin{array}{ll}
x=(u+1)(2u+5), & p=2u^3+12u^2+24u+15, \\
y=(u+2)(2u^2+8u+7), & q=(4u^4+28u^3+73u^2+87u+40)/2, \\
z=(2u^2+7u+4)(2u^2+7u+7)/2, &r=(4u^4+28u^3+71u^2+77u+30)/2.
\end{array}
\end{equation*}
or
\begin{equation*}
\begin{array}{ll}
x=(u+3)(2u+3), & p=2u^3+12u^2+24u+16, \\
y=(u+1)(2u^2+10u+13), & q=(4u^4+36u^3+121u^2+177u+92)/2, \\
z=(2u^2+9u+8)(2u^2+9u+11)/2, &r=(u+2)(u+3)(2u+3)(2u+5)/2,
\end{array}
\end{equation*}
are integers and are solutions of the system (\ref{eq1}).
\end{proof}

Using the parametric solutions we have obtained above we get
\begin{cor}
Let $f(X)=X(X+a)$, where $a\in\Z\setminus\{0\}$. Then there is
infinitely many positive integer solutions of the system
\begin{equation*}
(\star)\quad f(x)+f(y)=f(p),\quad f(y)+f(z)=f(q),\quad
f(z)+f(x)=f(r).
\end{equation*}
\end{cor}
\begin{proof}
It is clear that we can assume $a>0$. Now, note that if sextuple
$x,y,z,p,q,r$ is a solution of system (\ref{eq1}), then the sextuple
$ax,ay,az,ap,aq,ar$ is a solution of the system ($\star$).
\end{proof}

In the light of the above corollary we end this section with the
following
\begin{ques}
Let $f\in\Z[x]$ be a polynomial of degree two with two distinct
roots in $\C$. Is the system of equations
\begin{equation*}
f(x)+f(y)=f(p),\quad f(y)+f(z)=f(q),\quad f(z)+f(x)=f(r),
\end{equation*}
solvable in different positive integers ?
\end{ques}

\section{Rational solutions of the system $t_{x}+t_{y}=t_{p},\; t_{y}+t_{z}=t_{q},\; t_{z}+t_{x}=t_{r}$}\label{Section3}

In the view of the Theorem \ref{thm1} it is natural to state the
following
\begin{ques}\label{ques2}
Is the set of one-parameter polynomial solutions of the system {\rm
(\ref{eq1})} infinite?
\end{ques}
We suppose that the answer for this question is YES. Unfortunately,
we are unable to prove this. So, it is natural to ask if instead
polynomials we can find rational parametric solutions.
\begin{thm}\label{thm2}
There is three-parameter rational solution of the system {\rm
(\ref{eq1})}.
\end{thm}
\begin{proof}
Let $u,v,w$ be a parameters. Lat us note that the system (\ref{eq1})
is equivalent to the system
\begin{equation}\label{eq2}
\begin{cases}
y=u(p-x),\quad u(y+1)=p+x+1,\\
z=v(q-y),\quad v(z+1)=q+y+1,\\
x=w(r-z),\quad w(x+1)=r+z+1.
\end{cases}
\end{equation}
Because we are interested in rational solutions, we can look on the
system of equations (\ref{eq2}) as on the system of linear equations
with unknowns $x,y,z,p,q,r$. This system has a solution given by
\begin{align*}
x(u,v,w)=&\frac{(u-1)w(1+u-2v+2uv+v^2+uv^2+(-1-u+v^2+uv^2)w)}{(1-u^2)(v^2-1)(w^2-1)+8uvw},\\
y(u,v,w)=&\frac{u(v-1)(1-u+v-uv+(-2+2v)w+(1+u+v+uv)w^2)}{(1-u^2)(v^2-1)(w^2-1)+8uvw},\\
z(u,v,w)=&\frac{v(w-1)(1-2u+u^2-v+u^2v+(1+2u+u^2-v+u^2v)w)}{(1-u^2)(v^2-1)(w^2-1)+8uvw},\\
\end{align*}
and the quantities $p,q,r$ can by calculated from the identities
\begin{align*}
&p(u,v,w)=\frac{ux(u,v,w)+y(u,v,w)}{u},\\
&q(u,v,w)=\frac{vy(u,v,w)+z(u,v,w)}{v},\\
&r(u,v,w)=\frac{wz(u,v,w)+x(u,v,w)}{w}.
\end{align*}
\end{proof}

\begin{rem}
{\rm It is clear that the same reasoning can be used in order to
prove that the system
\begin{equation*}
f(x)+f(y)=f(p),\quad f(y)+f(z)=f(q),\quad f(z)+f(x)=f(r),
\end{equation*}
where $f\in\Z[x]$ is a polynomial of degree two with two different
roots, has rational parametric solution depending on three
parameters.

 }\end{rem}
\section{Rational solutions of the system $t_{x}+t_{y}=t_{p},\; t_{y}+t_{z}=t_{q},\;
t_{z}+t_{x}=t_{r},\;t_{x}+t_{y}+t_{z}=t_{s}$}\label{Section4}

We start this section with quite natural
\begin{ques}\label{ques3}
Is the system of equations
\begin{equation}\label{eq3}
t_{x}+t_{y}=t_{p},\quad t_{y}+t_{z}=t_{q},\quad
t_{z}+t_{x}=t_{r},\quad t_{x}+t_{y}+t_{z}=t_{s}
\end{equation}
solvable in integers?
\end{ques}

This question is mentioned in the very interesting book \cite[page
292]{Guy} and it is attributed to K. R. S. Sastry. In this book we
can find triple of integers $x=11,\;y=z=14$ (which was find by Ch.
Ashbacher). These number with $p=r=18,\;q=20,\;s=23$ satisfy the
system (\ref{eq3}). It is clear that the question about different
numbers $x,y,z,p,q,r,s$ which satisfy (\ref{eq3}) is more
interesting. Using the computer search we find some 7-tuples of
different integers which are the solutions of (\ref{eq3}). Our
results are contained in Table 2.

\begin{equation*}
\begin{tabular}{|c|c|c|c|c|c|c|}
  \hline
  $x$ & $y$ & $z$ & $p$ & $q$ & $r$ & $s$ \\
  \hline
  \hline
  230 & 741 & 870 & 776 & 1143 & 900 & 1166\\
  609 & 779 & 923 & 989 & 1208 & 1106 & 1353 \\
  714 & 798 & 989 & 1071 & 1271 & 1220 & 1458 \\
  1224 & 1716 & 3219 & 2108 & 3648 & 3444 & 3848 \\
  \hline
\end{tabular}
\end{equation*}
\begin{center}
{\sc Table 2}
\end{center}

It is quite likely that there is a polynomial solution of the system
(\ref{eq3}). However we are unable to find such a polynomials.

Now we use the parametric solutions obtained in theorem \ref{thm2}
to deduce the following

\begin{thm}\label{thm3}
The system of diophantine equations {\rm (\ref{eq3})} has infinitely
many rational solutions depending on two parameters.
\end{thm}
\begin{proof}
We know that the functions $x,y,z,p,q,r\in\Q(u,v,w)$ we have
obtained in the proof of the Theorem \ref{thm2} satisfied the system
(\ref{eq1}). So in order to find solutions of the system (\ref{eq3})
it is enough to consider the last equation
$t_{x}+t_{y}+t_{z}=t_{s}$. If we put calculated quantities $x,y,z$
into the equation $t_{x}+t_{y}+t_{z}=t_{s}$, use the identity
$8t_{s}+1=(2s+1)^2$ and get rid of denominators we get the equation
of quartic curve $C$ defined over the field $\Q(u,v)$
\begin{equation*}
C:\;h^2=a_{4}(u,v)w^4+a_{3}(u,v)w^3+a_{2}(u,v)w^2+a_{1}(u,v)w+a_{0}(u,v)=:h(w),
\end{equation*}
where
\begin{align*}
&a_{0}(u,v)=a_{4}(-u,v)=(u-1)^2(-1+u+2v+2uv-v^2+uv^2)^2,\\
&a_{1}(u,v)=a_{3}(-u,v)=4(u-1)(v^2-1)\Big(\frac{u^4-1}{u-1}(v^2+1)+2(u-1)(u^2+4u+1)v\Big),\\
&a_{2}(u,v)=4(1-10u^2 + u^4 )v^2 +
8(u^4-1)v(1+v^2)+2(3+2u^2+3u^4)(1+v^4).
\end{align*}

Because the polynomial $h\in\Q(u,v)[w]$ has not multiple roots the
curve $C$ is smooth. Moreover, we have $\Q(u,v)$-rational point on
$C$ given by
\begin{equation*}
Q=(0,(u-1)(-1+u+2v+2uv-v^2+uv^2)).
\end{equation*}
If we treat $Q$ as a point at infinity on the curve $C$ and use the
method described in \cite[page 77]{Mor} we conclude that $C$ is
birationally equivalent over $\Q(u,v)$ to the elliptic curve with
the Weierstrass equation
\begin{equation*}
E:\;Y^2=X^3-27f(u,v)X-27g(u,v),
\end{equation*}
where
\begin{equation*}
\begin{array}{ll}
  f(u,v)=&  u^4(v^8+1)+4u^2(u^4-1)v(v^6+1)+\\
           &  \quad(1+8u^2-22u^4+8u^6+u^8)v^2(1+v^4)+  \\
           &  \quad4(u^4-1)(u^4-3u^2-1)v^3(1+v^2)+ \\
           &  \quad2(3-16u^2+29u^4-16u^6+3u^8)v^4,\\
           \\
  g(u,v)=  &  (u^2(v^4+1)+2(u^4-1)v(v^2+1)+2(2-5u^2+2u^4)v^2)\times \\
           &  (-2f(u,v)+3(u^2-1)^2v^2(1+u^2-2v+2u^2v+v^2+u^2v^2)^2).
\end{array}
\end{equation*}

The mapping $\phi:\;C\ni (w,h) \mapsto (X,Y)\in E$ is given by
\begin{align*}
&w=a_{4}(u,v)^{-1}\Big(\Big(\frac{16a_{4}(u,v)^{\frac{3}{2}}Y-27d(u,v)}{24a_{4}(u,v)X-54c(u,v)}\Big)-\frac{a_{3}(u,v)}{4}\Big)\\
&h=a_{4}(u,v)^{-\frac{3}{2}}\Big(-\Big(\frac{16a_{4}(u,v)^{\frac{3}{2}}Y-27d(u,v)}{24a_{4}(u,v)X-54c(u,v)}\Big)^2+\frac{8a_{4}(u,v)X}{9}+c(u,v)\Big).
\end{align*}
We should note that the quantity
$a_{4}(u,v)^{\frac{3}{2}}=((u+1)(1+u+\ldots\;))^3$ is a polynomial
in $\Z[u,v]$, so our mapping is clearly rational. The quantities
$c,d\in\Z[u,v]$ are given below:
\begin{equation*}
\begin{array}{ll}
  c(u,v)= & -4(u+1)^2\times\Big(-u^2(u+1)^2(v^8+1)+ \\
          & \quad\; (u^2-1)(1-10u-2u^2-10u^3+u^4)v(v^6+1)+ \\
          & \quad\;2(u-1)^2(3-4u-6u^2-4u^3+3u^4)v^2(v^4+1)+ \\
          & \quad\;(u-1)^2(15+10u+2u^2+10u^3+15u^4)v^3(v^2+1) + \\
          & \quad\;2(10+20u-5u^2-46u^3-5u^4+20u^5+10u^6)v^4\Big)/3, \\
          \\
  d(u,v)= & 16(u-1)u(u+1)^4v(v^2-1)\times \\
          & \quad\;(-1+u+2v+2uv-v^2+uv^2)(1+u^2-2v+2u^2v+v^2+u^2v^2)\times \\
          & \quad\;(-1+u^2-4v-4u^2v+10v^2-10u^2v^2-4v^3-4u^2v^3-v^4+u^2v^4).
\end{array}
\end{equation*}

In order to finish the proof of our theorem we must show that the
set of $\Q(u,v)$-rational points on the elliptic curve $E$ is
infinite. This will be proved if we find a point with infinite order
in the group $E(\Q(u,v))$ of all $\Q(u,v)$-rational points on the
curve $E$. In general this is not an easy task. First of all note
that there is torsion point $T$ of order 2 on the curve $E$ given
by
\begin{equation*}
T=(3u^2(v^4+1)+6(u^4-1)v(v^2+1)+6(u^2-2)(2u^2-1)v^2,\;0).
\end{equation*}
It is clear that this point is not suitable for our purposes.
Fortunately in our case we can find another point
\begin{align*}
P=\Big(&\frac{3}{4}((3-2u^2+3u^4)(v^4+1)+8(u^4-1)v(v^2+1)+2(5-14u^2+5u^4)v^2),\\
       &\frac{27}{8}(u^2-1)^2(v^2-1)((u^2+1)(v^4+1)+4(u^2-1)v(v^2+1)+6(u^2+1)v^2\Big).
\end{align*}

Now, if we specialize the curve $E$ for $u=2,v=3$, we obtain the
elliptic curve
 \begin{equation*}
E_{2,\;3}:\;Y^2=X^3-28802736X+40355763840
 \end{equation*}
with the point $P_{2,3}=(5736,\;252720)$, which is the
specialization of the point $P$. As we know, the points of finite
order on the elliptic curve $y^2=x^3+ax+b,\;a,\;b\in\mathbb{Z}$ have
integer coordinates \cite[page 177]{Sil}, while
\begin{equation*}
2P_{2,\;3}=(765489/100,\; -518102487/1000);
\end{equation*}
therefore, $P_{2,\;3}$ is not a point of finite order on
$E_{2,\;3}$, which means that $P$ can not be a point of finite order
on $E$. Therefore, $E$ is a curve of positive rank. Hence, its set
of $\Q(u,v)$-rational points is infinite and our theorem is proved.
\end{proof}

Let us note the obvious
\begin{cor}
Let $f\in\Z[x]$ be a polynomial of degree two with two distinct
rational roots. Then the system of equations
\begin{equation*}
f(x)+f(y)=f(p),\quad f(y)+f(z)=f(q),\quad f(z)+f(x)=f(r),\quad
f(x)+f(y)+f(z)=f(s),
\end{equation*}
has infinitely many rational parametric solutions depending on two
parameters.
\end{cor}
\begin{exam}
{\rm Using the method of proof of the above theorem we produce now
an example of rational functions $x,y,z\in\Q(u)$ which satisfy
system (\ref{eq3}) for some $p,q,r,s\in\Q(u)$ which can be easily
find (with the computer of course). Because the considered
quantities are rather huge we put here $v=2$. Then we have that
\begin{equation*}
P_{2}+T_{2}=(1404u^4+219u^2+4,\;8(9u^2+1)^2(81u^2+1)),
\end{equation*}
where $P_{2},T_{2}$ are specializations of points $P,T$ in $v=2$
respectively. Now, we have that
\begin{align*}
&(w,h)=\phi^{-1}(P_{2}+T_{2})=\\
      &\Big(\frac{(9u-1)^2(9u+1)(63u^2+17)}{3(u+1)(6561u^4+1134u^3+306u-1)},\;\frac{8(9u-1)(9u+1)F(u)}{9(u+1)^2(6561u^4+1134u^3+306u-1)^2}\Big)
\end{align*}
where
\begin{align*}
F(u)=&\;43046721u^8+11573604u^7+6388956u^5+\\
     &\quad 1285956u^6+680886u^4+919836u^3+93636u^2+10404u+1.
\end{align*}

Using the calculated values and the definition of $x,y,z$ given in
the proof of Theorem \ref{thm2} we find that the functions
\begin{align*}
&x(u)=\frac{2(u-1)(9u-1)^2(9u+1)^2(63u^2+17)(81u^2+1)}{G(u)},\\
&y(u)=\frac{3u(11+42u^2+2187u^4)(1+2754u^2+3645u^4)}{G(u)},\\
&z(u)=\\
&\quad \frac{2(27u^2-18u-5)(81u^2-48u-1)(135u^2+18u+7)(243u^3-99u^2+57u-1)}{3G(u)},\\
\end{align*}
where
\begin{align*}
G(u)=-&23914845u^9+110008287u^8-18528264u^7+15956352u^6+\\
      &-473850u^5-940410u^4-91008u^3-96264u^2-33u+35
\end{align*}
satisfied the system of equations (\ref{eq3}) }\end{exam}

\bigskip

 \hskip 4.5cm       Maciej Ulas

 \hskip 4.5cm       Jagiellonian University

 \hskip 4.5cm       Institute of Mathematics

 \hskip 4.5cm       Reymonta 4

 \hskip 4.5cm       30 - 059 Krak\'{o}w, Poland

 \hskip 4.5cm      e-mail:\;{\tt Maciej.Ulas@im.uj.edu.pl}

 \end{document}